\documentclass[a4paper,10pt]{article}
\usepackage{amssymb}
\usepackage{amsmath,amsfonts,amsthm,amssymb}
\usepackage[dvips]{graphics}
\usepackage{epsfig}
\usepackage{indentfirst}
\usepackage{cancel,soul}

\usepackage{color}
\usepackage{bm}
\allowdisplaybreaks

\usepackage[colorlinks=true,citecolor=red,linkcolor=blue,%
urlcolor=RubineRed,pdfpagetransition=Blinds,pdftoolbar=false,pdfmenubar=false]{hyperref}

\newtheoremstyle{theorem}
  {10pt}          
  {10pt}  
  {\sl}  
 {}
  {\bf}  
  {. }    
  { }    
  {}     
\theoremstyle{theorem}

\newtheorem{theorem}{Theorem}[section]

\newtheorem{definition}{Definition}[section]
 
 \newtheorem{proposition}{Proposition}[section]
 \newtheorem{remark}{Remark}[section]
  
\numberwithin{equation}{section}

\newtheoremstyle{defi}
  {10pt}          
  {10pt}  
  {\rm}  
  {}  
  {\bf}  
  {. }    
  { }    
  {}     
\theoremstyle{defi}

\begin{document}
\baselineskip = 16pt

\title{\bf Energy Conservation and Vanishing Viscosity Limit for the Primitive Equations}

\author{\v{S}\'{a}rka Ne\v{c}asov\'{a}$^1$
\footnote{Email: matus@math.cas.cz} \ \ \
Tong Tang$^{2}$ \footnote{Email: tt0507010156@126.com}\\
Emil Wiedemann $^{3}$\footnote{Email: emil.wiedemann@fau.de}\\
Lu Zhu $^{4}$\footnote{Email: zhulu@hhu.edu.cn}\\
{\small  1. Institute of Mathematics of the Academy of Sciences of the Czech Republic,} \\
{\small \v Zitn\' a 25, 11567, Praha 1, Czech Republic}\\
{\small 2. School of Mathematical Science}\\
{\small Yangzhou University, Yangzhou 225002, P.R. China}\\
{\small 3. Department of Mathematics}\\
{\small Friedrich-Alexander-Universit\"at Erlangen-N\"{u}rnberg, Cauerstr.\ 11,
91058 Erlangen, Germany}\\
{\small 4. School of Mathematics}\\
{\small Hohai University, Nanjing, China}\\
\date{}}

\maketitle

\begin{abstract}
In this paper, we consider the problem of energy conservation for weak solutions of the inviscid Primitive Equations (PE) in a bounded domain. Based on the work [Bardos et al., Onsager's conjecture with physical boundaries and an application to the vanishing viscosity limit, Comm. Math. Phys., 2019, 291-310], we prove the energy conservation for PE with boundary condition under suitable Onsager-type assumptions. But due to the special structure of PE system and its domain, some new challenging difficulties arise: the lack of information about the vertical velocity, and existing corner points in the domain. We introduce some new ideas to overcome the above obstacles. As a byproduct, we give a sufficient condition for absence of anomalous energy dissipation in the vanishing viscosity limit.

\vspace{0.5cm}


\medskip

\end{abstract}

\maketitle
\section{Introduction}\setcounter{equation}{0}
It is needless to emphasize the importance of atmospheric and ocean flows both in science and everyday life. From the mathematical and numerical points of view however, it is very complicated to use the full hydrodynamic and thermodynamic equations to describe the motion of geophysical fluids and thus describe meteorological and oceanographic phenomena. As a simplified model, therefore, Primitive Equations (PE) were introduced in geophysical fluid dynamics. As a model for long-term weather prediction, the PE system has been extensively studied theoretically, numerically, and experimentally. It is derived from the Navier-Stokes equations or Euler equations using the Boussinesq approximation and the hydrostatic approximation. In this paper, we investigate the inviscid Primitive Equations (PE) system:
\begin{equation}\label{1.1}
\left\{
\begin{array}{llll}
\partial_{t}\mathbf u+\text{div}_{\bf x}(\mathbf{u}\otimes\mathbf{u})+\partial_z(\mathbf u w)+\nabla_{\bf x}p=0, \\
\partial_zp=0,\\
\text{div}_{\bf x}\mathbf u+\partial_zw=0,
\end{array}\right.
\end{equation}
where $\mathbf u(t,{\bf x},z)\in\mathbb{R}^2$, $w(t,{\bf x},z)\in\mathbb{R}$, and $p$ represent the horizontal velocity, vertical velocity, and pressure respectively.
We denote the velocity as $\mathbf{U}=(\mathbf u,w)$ and the space variable as ${\bf X}=({\bf x}, z)$,
where ${{\bf x} \in \mathbb{R}^2}$ denotes the horizontal and $z$ the vertical direction. The domain is a cylinder, that is
\begin{align*}
\Omega=\{S\times[0,1]:x\in S, z\in[0,1]\},
\end{align*}
where $S\subseteq\mathbb{R}^2$ is a bounded domain with $\partial S\in C^2$. The boundary is $\partial\Omega=\Gamma_b\cup\Gamma_t\cup\Gamma_s$, with
\begin{align}\label{1.2}
&\text{the bottom}\hspace{5pt}\Gamma_b=\{({\bf x},z), \mathbf x\in S,z=0\},\nonumber\\
&\text{the top}\hspace{5pt}\Gamma_t=\{({\bf x},z), \mathbf x\in S,z=1\},\nonumber\\
&\text{the sidewall}\hspace{5pt}\Gamma_s=\{({\bf x},z), \mathbf x\in\partial S,0<z<1\}.
\end{align}
We assume the slip condition for $\mathbf U$, that is
\begin{align}\label{1.3}
\mathbf U\cdot\mathbf N|_{\partial\Omega}=0,
\end{align}
where $\mathbf N$ is the unit outward normal vector to the boundary. On the bottom and top, this condition translates to $w|_{z=0,1}=0$, while on the sidewall the boundary condition is $\mathbf u\cdot\mathbf n|_{\Gamma_s}=0$, where $\mathbf n$ is the normal vector to $\partial S$.

For the three-dimensional incompressible Euler equations, well-posedness is known only locally in time and in function spaces of high regularity (essentially anything better than $C^1$).
Evidently, the conservation of energy for the Euler equations is valid for such smooth (or strong) solutions. In contrast, weak solutions, which are better suited to describe turbulent flows, need not conserve the energy. This is phenomenologically and experimentally expected in the study of turbulence. In terms of rigorous analysis, the Scheffer-Shnirelman \cite{sc,sh} construction yields a first counterexample to energy conservation (and uniqueness) of weak solutions of Euler. It is therefore interesting to investigate suitable additional conditions on the velocity field for weak solutions such as to obtain energy conservation.

Lars Onsager formulated a famous conjecture when in his work on turbulence. He conjectured that the {incompressible Euler equations} conserve energy if the velocity
$\mathbf u\in L^3((0,T);C^{0,\alpha}(\mathbb{R}^3))$ with $\alpha>\frac{1}{3}$, and may dissipate energy below this threshold. Both parts of Onsager's conjecture have been a lively direction of recent research. The first (conservative) part was proved by Constantin, E and Titi in \cite{co} (also by  Eyink \cite{ey}, and locally in the work of Duchon and Robert \cite{du}), stating that if ${\bf u}$ belongs to the Besov space $L^3([0,T];B^{\alpha,\infty}_3(\mathbb{T}^3))\cap C([0,T];L^2(\mathbb{T}^3))$ with $\alpha>\frac{1}{3}$,
then the energy is conserved. Cheskidov et al~ \cite{ch} and Fjordholm and Wiedemann \cite{f} sharpened this result to optimal Besov-type spaces at the critical fractional differentiability exponent $1/3$. Similar results were obtained for the incompressible {\it inhomogeneous} and also the compressible isentropic Euler equations by Feireisl et al.~\cite{e4} (see also~\cite{lsh}).

The second part of the conjecture says there exist weak solutions of the Euler equation for $\alpha<\frac{1}{3}$ which do not conserve energy. This part has been solved by by Isett \cite{i}, Buckmaster et. al. \cite{bu}, relying on a series of breakthrough papers of De Lellis and Sz\'{e}kelyhidi \cite{de1,de2}, who developed the so-called convex integration technique in the realm of fluid dynamics.

Regarding the PE system, the mathematical study of well-posedness was initiated in the 1990s by J.-L.~Lions, Temam, and Wang \cite{l1,l2}. Then Guill\'en-Gonz\'alez, Masmoudi, and Rodr\'{\i}guez-Bellido \cite{gu} proved the local existence of strong solutions and their uniqueness using anisotropic estimates. A key milestone was achieved by Cao and Titi \cite{c1}, who proved the global well-posedness of PE in the three dimensional case. Then, by virtue of semigroup methods, Hieber and Kashiwabara \cite{h} improved this result by relaxing the smoothness of the initial data. Recently, Li and Liang \cite{li} proved the global well-posedness of PE with free boundary condition. However, to date there have been no results regarding energy conservation for PE system save those by Boutros, Markfelder, and Titi \cite{bo}, who proposed three types of weak solution, and in each case showed the energy conservation for the incompressible inviscid PE in $\mathbb{T}^3$.

However, we should emphasize that all the above results about energy conservation concern on the whole space or the torus, where there are no physical boundaries. When passing from local to global energy conservation, it is hard to control the regularity of the velocity field near the boundary. Therefore, it is a natural question to ask whether we can obtain an Onsager-type result on a bounded domain with slip boundary condition.
For the incompressible Euler equations, the first result in the special case of a half-space is due to Robinson et al.~\cite{ro} via a reflection technique. On general bounded domains, the problem was tackled by Bardos, Titi and Wiedemann \cite{b2}, who also give an application to the vanishing viscosity limit. The basic strategy used there \cite{b2} is localization and regularity of the pressure via elliptic theory. Thus, in order to obtain the global energy conservation, one needs first to get the local version and in addition suitable estimates near the boundary. See also Drivas and Nguyen \cite{dr} for related results.

In the context of geophysical fluid dynamics, the boundary condition is necessary and essential. Based on the work of \cite{b2}, we will investigate the energy conservation for PE with boundary condition. Some techniques in this paper are inspired by methods introduced in \cite{b2}. However, their extension to the PE system is non-standard, since PE offers a strikingly different picture. The increased difficulty of PE compared to Euler is at least twofold: First, there is a lack of information on the vertical velocity in the momentum equation. Secondly, the boundary in \cite{b2} is assumed $C^2$, which is an essential ingredient in the proof. However, the domain in the current note is a cylinder, which exhibits edges at the top and the bottom. It is therefore evident that the method used in \cite{b2} can not be applied directly to the PE system. Inspired by \cite{bo}, who work in anisotropic Besov spaces, we use an anisotropic H\"{o}lder space to solve the first difficulty. For the second difficulty, we construct an approximate interior domain (see Figure 2) with $C^2$ boundary, and obtain the corresponding result, then pass to the limit. {To the authors' knowledge, our work is the first result about energy conservation for PE system with boundary condition.}

Unlike most systems in fluid dynamics, for PE there is no canonical way to define weak solutions -- this is why Boutros et al.~\cite{bo} propose three different notions. Here, we choose the simplest one for convenience. We believe it would be possible to work with the other two notions as well to achieve results comparable with those presented here. Also, the choice of H\"older instead of Besov spaces is merely for simplicity and convenience.

The paper is organized as follows:
In Section \ref{S2}, we introduce notation and give some useful propositions, then state the main theorems.
Section \ref{S3} is devoted to the proof of local energy conservation. Section \ref{S3} focuses on the global energy conservation. Finally, Section \ref{S5} shows the result on the vanishing viscosity limit.
\vskip 0.5cm

\section{Preliminaries} \label{S2}

\subsection{Definition of weak solutions for PE }
\begin{definition}

A triple of $(\mathbf u, w, p)$ is called a \emph{weak solution} to the PE system \eqref{1.1} if

\noindent
$\bullet$  $\mathbf u\in L^2((0,T);L^2(\Omega))$, $w\in L^2((0,T);L^2(\Omega))$, $p\in \mathcal{D}'((0,T)\times\Omega)$;

\noindent
$\bullet$
all the equations are satisfied in the sense of distributions, i.e.
\begin{align}
&\int^T_0\int_{\Omega}\mathbf u\partial_t\mathbf\Psi d\mathbf xdzdt+\int^T_0\int_{\Omega}\mathbf{u}\otimes\mathbf{u}:\nabla_{\mathbf x}\mathbf\Psi d\mathbf xdzdt\nonumber\\
&\hspace{5pt}+\int^T_0\int_\Omega\mathbf uw\partial_z\mathbf\Psi d\mathbf xdzdt
+\int^T_0\int_\Omega p\operatorname{div}_{\mathbf x}\mathbf\Psi d\mathbf xdzdt=0,\label{2.1}\\
&\int^T_0\int_\Omega p\partial_z\psi d\mathbf xdzdt=0,\label{2.2}\\
&\int^T_0\int_\Omega (\mathbf u\nabla_{\mathbf x}\psi+w\partial_z\psi )d\mathbf xdzdt=0,\label{2.3}
\end{align}
for all test functions $\Psi\in \mathcal{D}((0,T)\times\Omega;\mathbb{R}^3)$, $\psi\in \mathcal{D}((0,T)\times\Omega;\mathbb{R})$.
\end{definition}

\begin{remark}
Chiodaroli and Mich\'{a}lek \cite{chi} used convex integration to show the existence of infinitely many global weak solutions for inviscid PE on a bounded domain for certain initial data.

\end{remark}

\subsection{Notation}

Here we adopt a framework similar as in~\cite{b2}. Although rigorously meaningless, we sometimes attach a subscript to $\mathbb{R}^m$ to clarify which variable is being referred to, for instance $\mathbb{R}^3_{\mathbf X}$ or $\mathbb{R}^4_{t,\mathbf X}=\mathbb{R}_t\times\mathbb{R}^3_{\mathbf X}$.

For an open set $\widetilde{Q}\subset\mathbb{R}^m$ and a given test function $\Psi\in\mathcal{D}(\widetilde{Q})$, we know $\Psi$ is compactly supported in $\widetilde{Q}$ by definition and denote its support by $S_\Psi$. Then we can introduce the following open sets by choosing sufficiently small $\eta>0$:
\begin{align}\label{2.4}
&S_\Psi\subset\subset Q_3\subset\subset Q_2\subset\subset Q_1\subset\subset \widetilde{Q},\nonumber\\
&d(S_\Psi,\mathbb{R}^m\setminus Q_3)>\eta,\hspace{3pt}d(Q_3,\mathbb{R}^m\setminus Q_2)>\eta,\hspace{3pt}
d(Q_2,\mathbb{R}^m\setminus Q_1)>\eta,\hspace{3pt}d(Q_1,\mathbb{R}^m\setminus \widetilde{Q})>\eta,\hspace{3pt}
\end{align}
where $d(A,B)$ means the distance between two nonempty sets $A$ and $B$. Next,
consider a smooth function $I_{2,\eta}\in\mathcal{D}(\mathbb{R}^m)$ such that
\begin{equation*}
I_{2,\eta}(z)=
\left\{
\begin{array}{llll}
1, \hspace{5pt} z\in Q_2,\\
0, \hspace{5pt} z\notin Q_1.
\end{array}\right.
\end{equation*}
For any distribution $T\in\mathcal{D}(\widetilde{Q})$, we can consider the distribution $\overline{T}:=I_{2,\eta}T$. Note that this extension procedure is contingent on the choices of $\Psi$ and $I_{2,\eta}$.

The following is trivial, but a detailed proof can be consulted from~\cite[Proposition 3.2]{b2} anyway:


\begin{proposition}\label{pro2.1} [\cite[Proposition 3.2]{b2}]

For a scalar or tensor-valued function $h\in L^p(Q_1)$ $(1\leq p\leq\infty)$ and for any $f\in C^\infty$ satisfying $f(0)=0$ and
$|f(\xi)|\leq C(1+|\xi|^p)$, write $\overline{f}:=I_{2,\eta}f$.

Then, if $\Psi\in\mathcal{D}(\mathbb{R}^m)$ with support in $Q_3$,
\begin{align}\label{2.5}
\langle\langle\overline{f}(h),\Psi\rangle\rangle=\langle\langle f(\overline{h}),\Psi\rangle\rangle.
\end{align}
Furthermore, for any multi-order derivative $D^\alpha$,
\begin{align}\label{2.6}
\langle\langle D^\alpha\overline{f(h)},\Psi\rangle\rangle=\langle\langle D^\alpha f(\overline{h}),\Psi\rangle\rangle.
\end{align}
Finally, for every $\sigma>0$ small enough such that
\begin{align*}
0<\sigma<\frac{\eta}{2}<\frac{d(Q_3,\mathbb{R}^m\setminus\overline{Q}_2)}{2},
\end{align*}
one has
\begin{align}\label{2.7}
\langle\langle \rho_\sigma\star\overline{f(h)},\Psi\rangle\rangle=\langle\langle\rho_\sigma\star f(\overline{h}),\Psi\rangle\rangle.
\end{align}
\end{proposition}

\emph{Mutatis mutandis}, we can consider a similar setting in time-space cylindrical domains. Let $\widetilde{Q}=(t_1,t_2)\times\widetilde{\Omega}\subset\subset(0,T)\times\Omega$. For a given function $\psi\in \mathcal{D}(\widetilde{Q})$ with support contained in $(t_a,t_b)\times \Omega_\psi$, we choose subsets such that, for appropriate $\eta>0$,
\begin{align}\label{2.8}
&\Omega_\Psi\subset\subset \Omega_3 \subset\subset \Omega_2\subset\subset \Omega_1\subset\subset \widetilde{\Omega}\subset\subset\Omega,\nonumber\\
&d(\Omega_\Psi,\Omega\setminus \Omega_3)>\eta,\hspace{3pt}d(\Omega_3,\Omega\setminus \Omega_2)>\eta,\hspace{3pt}
d(\Omega_2,\Omega\setminus \Omega_1)>\eta,\hspace{3pt}d(\Omega_1,\Omega\setminus \widetilde{\Omega})>\eta.
\end{align}
Moreover one can choose $\tau>0$ small enough such that Proposition \ref{pro2.1} can be applied to the open sets
\begin{align}\label{2.9}
(t_a,t_b)\times \Omega_\Psi&\subset\subset Q_3=(t_1+3\tau,t_2-3\tau)\times\Omega_3\subset\subset Q_2=(t_1+2\tau,t_2-2\tau)\times\Omega_2\nonumber\\
&\subset\subset Q_1=(t_1+\tau,t_2-\tau)\times\Omega_1\subset\subset \widetilde{Q}=(t_1,t_2)\times\widetilde{\Omega}.
\end{align}
Similarly as before, any distribution $T\in\mathcal{D}'(\widetilde{Q})$ is extended as a distribution in $\mathcal{D}'(\mathbb{R}_t\times\mathbb{R}^3_{\mathbf X})$ by $\overline{T}=I_{2,\sigma}T$ with $I_{2,\sigma}(\mathbf X,t)=I_{2,\tau}(t)I_{2,\eta}(\mathbf X)$, where $I_{2,\tau}\in\mathcal{D}(\mathbb{R}_t)$ and $I_{2,\eta}\in\mathcal{D}(\mathbb{R}^3_{\mathbf X})$ satisfy
\begin{equation*}
I_{2,\tau}(t)=
\left\{
\begin{array}{llll}
1, \hspace{5pt} t\in (t_1+2\tau,t_2-2\tau),\\
0, \hspace{5pt} t\notin (t_1+\tau,t_2-\tau),
\end{array}\right.
\end{equation*}
and
\begin{equation*}
I_{2,\eta}(\mathbf X)=
\left\{
\begin{array}{llll}
1, \hspace{5pt} \mathbf X\in \Omega_2,\\
0, \hspace{5pt} \mathbf X\notin \Omega_1.
\end{array}\right.
\end{equation*}
Moreover, for $\kappa\in(0,\frac{\tau}{2})$, $\epsilon\in(0,\frac{\eta}{2})$ and $\sigma=(\kappa,\epsilon)$ we introduce the mollifiers $\rho_\sigma$ as
\begin{align}\label{2.10}
\rho_\sigma(t,\mathbf X)=\rho_{\kappa,\epsilon}(t,\mathbf X)=\frac{1}{\kappa}\rho\left(\frac{|t|}{\kappa}\right)\frac{1}{\epsilon^3}\rho\left(\frac{|\mathbf X|}{\epsilon}\right),
\end{align}
where $\rho\in \mathcal{D}(\mathbb{R})$ is a standard mollifier.

Next, for $f\in\mathcal{D}'(\mathbb{R}^3_{\mathbf X})$ and $F\in\mathcal{D}'(\mathbb{R}_t\times\mathbb{R}^3_{\mathbf X})$, we use the notation
\begin{align}\label{2.11}
(f)^\epsilon=\rho_\epsilon\star f,\hspace{5pt} (F)^\sigma=\rho_\sigma\star F=\rho_\kappa\star_t\rho_\epsilon\star_{\mathbf X}F.
\end{align}

As $\partial S\in C^2$, there exists a neighbourhood of $\partial S$ which enjoys the \emph{unique nearest point property}~\cite{foote}, that is, for every $\bf x\in S$ sufficiently close to the boundary, there exists a unique point $\bm{\sigma}(\bf x)\in\partial S$ such that
\begin{equation*}
|\bm{\sigma}(\bf x)-\bf x|=\inf_{\bf y\in\partial S}|\bf x-\bf y|.
\end{equation*}
By abuse of notation, for $\bf X\in \Omega$ sufficiently close to $\partial S\times [0,1]$ we sometimes also write $\bf\sigma(\bf X)$ to denote the unique closest point to $\bf X$ in $\partial S\times [0,1]$.

Finally, we introduce the anisotropic H\"{o}lder seminorm and define $(0<\alpha,\beta<1)$
\begin{align}\label{2.12}
\|\mathbf u\|_{C^{0,\alpha}_{\mathbf x}C^{0,\beta}_{z}}=\sup_{\xi_z\in\mathbb{R}\setminus\{0\}}\frac{|\mathbf u(\mathbf x,z+\xi_z)-\mathbf u(\mathbf x,z)|}{|\xi_z|^\beta}
+\sup_{\xi_h\in\mathbb{R}^2\setminus\{0\}}\frac{|\mathbf u(\mathbf x+\xi_h,z)-\mathbf u(\mathbf x,z)|}{|\xi_h|^\alpha}.
\end{align}

Now we are ready to show our results.


\subsection{Main results}
First, we give the local version of energy conservation for PE.

\begin{theorem}\label{t2.1}
Let $(\mathbf u, w,p)$ 
be a weak solution of PE \eqref{1.1} in an open subset $\widetilde{Q}=(t_1,t_2)\times\widetilde{\Omega}\subset\subset(0,T)\times\Omega$ that satisfies the following two conditions:

1. There exists $\gamma>0$ so that for $V_\gamma=\{\mathbf X\in\widetilde{\Omega}: d(\mathbf X,\partial\widetilde{\Omega})<\gamma\}$, there are $M_0(V_\gamma)>0$ and $\theta(V_\gamma)>0$ such that
\begin{align}\label{2.13}
p\in L^{\frac{3}{2}}\big{(}(t_1,t_2);H^{-\theta(V_\gamma)}(V_\gamma)\big{)}\leq M_0(V_\gamma)<\infty;
\end{align}

2. There exist $\alpha\in (1/2,1)$ and {$\beta\in (1/2,1)$ }such that $2\operatorname{min}\{\alpha,\beta\}+\operatorname{max}\{\alpha,\beta\}>2$ (the relation between $\alpha$ and $\beta$ can be seen in Figure 1) and $M(\widetilde{Q})>0$ such that
\begin{align}\label{2.14}
\int^{t_2}_{t_1}\|\mathbf u\|^3_{C^{0,\alpha}_{\mathbf x}C^{0,\beta}_{z}(\overline{\widetilde{\Omega}})}dt\leq M(\widetilde{Q})<\infty.
\end{align}
Then $(\mathbf u, w, p)$ satisfies in $\widetilde{Q}$ the local energy conservation
\begin{align*}
\partial_t\frac{|\mathbf u|^2}{2}+\operatorname{div}_{\bf x}\left( \left(\frac{|\mathbf u|^2}{2}+p\right)\mathbf u\right)+\partial_z\left(\left(\frac{|\mathbf u|^2}{2}+p\right)w\right)=0
\hspace{5pt}\text{in}\hspace{5pt}\mathcal{D}'((t_1,t_2)\times\widetilde{\Omega}).
\end{align*}

\end{theorem}

Based on the local result, we will prove a global version. For this we need to consider, for sufficiently small $\eta>0$, the $\eta$-neighbourhood of the nonsmooth part of $\partial\Omega$:

\begin{equation*}
\Gamma_\eta:=\{\bm{X}\in\Omega: d(x,\overline{\Gamma_s}\cap(\Gamma_t\cup\Gamma_b))<\eta\}.
\end{equation*}

\begin{theorem}\label{t2.2}
Let $(\mathbf u, w,p)$ 
be a weak solution of PE \eqref{1.1} in $(0,T)\times\Omega$ satisfying the following hypotheses:

1. For some $\eta_0>0$ and $\eta>0$ sufficiently small,
\begin{align}
&p\in L^{\frac{3}{2}}\big{(}(0,T);H^{-\theta(V_{\eta_0})}(V_{\eta_0})\big{)},\label{2.15}\\
&\lim_{\eta\rightarrow0}\int^T_0\frac{1}{\eta}\left\{
\int_{[S_{\frac{5}{4}\eta}\setminus S_{\frac{3}{2}\eta}]
\times[0,1]}\left|\left(\frac{|\bm{u}|^2}{2}+p\right)\bm{u}\cdot
\bm{n}(\bm{\sigma}(\bm{X}))\right|
d\bm{X}+\int_{S\times I_\eta}
\left|\left(\frac{|\bm{u}|^2}{2}+p\right)w\right|
d\bm{X}
\right\}dt=0,\label{2.16}
\end{align}
where $S_{\eta}=\{\bm{x}\in S:\ {\rm dist}_{{\mathbb R}^2}(\bm{x},\partial S)
<\eta\}$, $I_{\eta}=(\eta,2\eta)\cup(1-2\eta,1-\eta)$, and $V_\eta$ is the $\eta$-neighbourhood of $\Omega$;

2. For every open set $\widetilde{Q}=(t_1,t_2)\times\widetilde{\Omega}\subset\subset(0,T)\times\Omega$, there exist $\alpha\in (1/2,1)$ and {$\beta\in (1/2,1)$ } such that $2\operatorname{min}\{\alpha,\beta\}+\operatorname{max}\{\alpha,\beta\}>2$ with the property
\begin{align}\label{2.17}
\int^{t_2}_{t_1}\|\mathbf u\|^3_{C^{0,\alpha}_{\mathbf x}C^{0,\beta}_{z}(\overline{\widetilde{\Omega}})}dt\leq M(\widetilde{Q})<\infty.
\end{align}

3. There exist some $\mu_1,\mu_2\in\mathbb{R}$ such that
\begin{align}\label{2.18}
&\|p\|_{L^\frac{3}{2}((0,T)\times\Gamma_\eta)}\leq C \eta^{\mu_1},\hspace{8pt}
\|\bm{\bm{U}}\|_{L^3((0,T)\times \Gamma_\eta)}\leq C\eta^{\mu_2},\nonumber\\
&\mu_1+\mu_2>1,\ \mu_2>\frac{1}{3}.
\end{align}
Then $(\mathbf u, w, p)$ globally conserves the energy, i.e.,
\begin{align*}
\|\mathbf u(t_2)\|_{L^2(\Omega)}=\|\mathbf u(t_1)\|_{L^2(\Omega)}
\end{align*}
for any $0<t_1<t_2<T$.
\end{theorem}
\begin{remark}
In our case, $\Omega$ is only a Lipschitz domain and the usual global estimates for elliptic equations can not be applied directly. Thus we construct a smooth subdomain $\Omega_\eta$ (defined in \eqref{4.3}) of $\Omega$ whose distance to $\partial\Omega$ is $O(\eta)$. Then we obtain the local energy conservation in $\Omega_\eta$ from Theorem \ref{t2.1}. Therefore by taking $\eta\rightarrow0$ and under condition \eqref{2.18}, we can prove Theorem \ref{t2.2}.

\end{remark}

\begin{figure}\label{regcyl}
\begin{minipage}[t]{0.48\textwidth}
\centering
\includegraphics[scale=0.2]{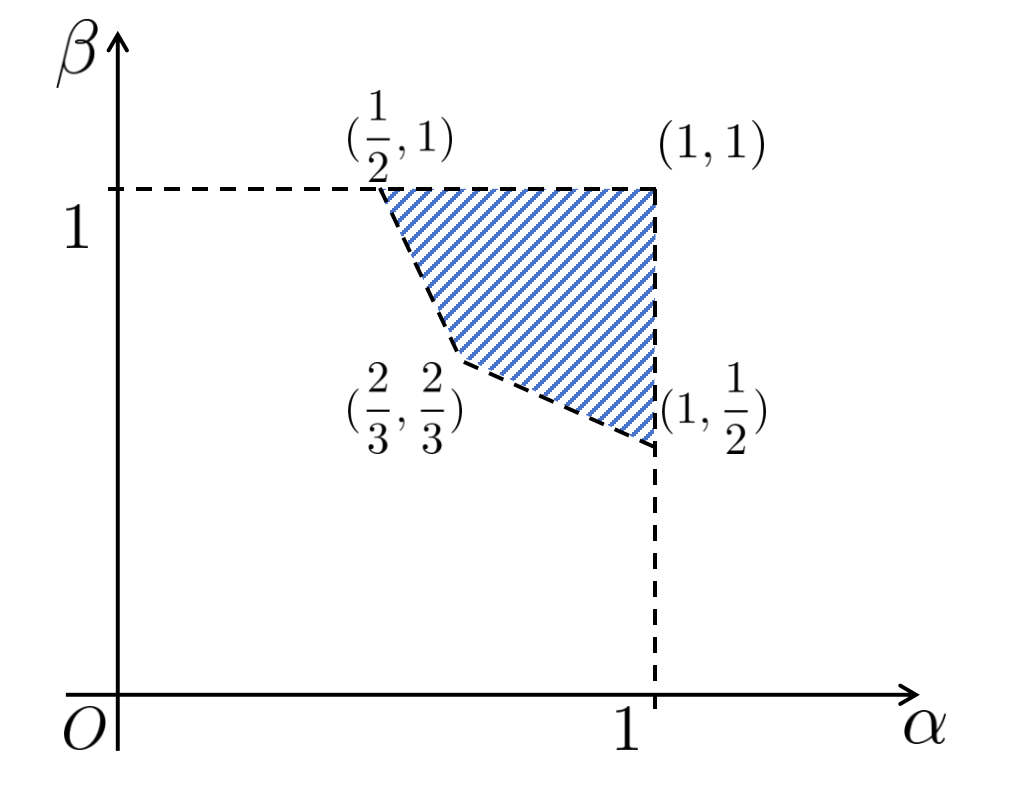}
\caption{relation of $\alpha$ and $\beta$}
\end{minipage}
\centering
\begin{minipage}[t]{0.28\textwidth}
\centering
\includegraphics[scale=0.2]{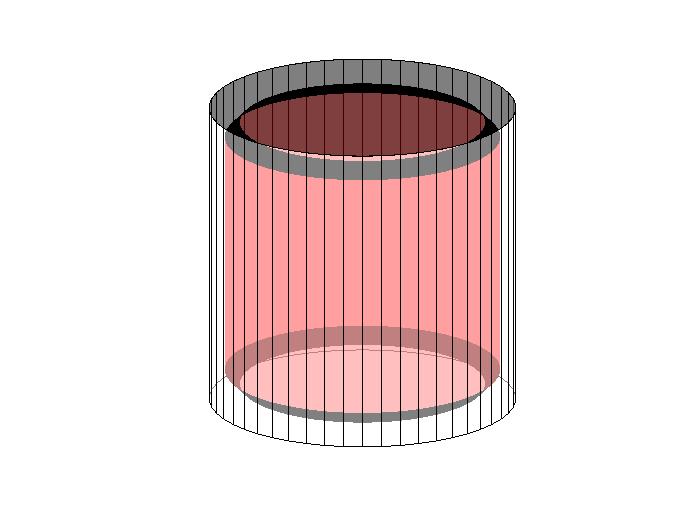}
\caption{regularize the cylinder}
\end{minipage}
\end{figure}


\section{Proof of Theorem \ref{t2.1}}\label{S3}

Before proving Theorem \ref{t2.1}, we need to obtain the local estimate on the pressure.
\begin{proposition}\label{pressureest}
Let $(\mathbf u, w,p)$ be a weak solution of PE \eqref{1.1} that satisfies in $\widetilde{Q}=(t_1,t_2)\times\widetilde{\Omega}\subset\subset(0,T)\times\Omega$ the hypothesis of Theorem~\ref{t2.1}. Then the restriction of the pressure $p$ to $Q_2=(t_1,t_2)\times\Omega_2$ belongs to $L^{\frac{3}{2}}\big{(}(t_1,t_2);C^{0,\alpha}_{\bf x}C^{0,\beta}_z(\overline{\Omega_2})\big{)}$ and satisfies the estimate
\begin{align*}
\|p\|_{L^{\frac{3}{2}}\left((t_1,t_2);C^{0,\alpha}_{\bf x}C^{0,\beta}_z(\overline{\Omega_2})\right)}\leq C
\left( \|\mathbf u\|_{L^{3}\left((t_1,t_2);C^{0,\alpha}_{\bf x}C^{0,\beta}_z(\overline{\Omega_2})\right)},\hspace{5pt}
\|p\|_{L^{\frac{3}{2}}\big{(}(t_1,t_2);H^{-\theta}(V_\gamma)\big{)}}\right).
\end{align*}

\end{proposition}

\begin{proof}
We integrate $\eqref{1.1}_1$ in the vertical direction,
\begin{align*}
\partial_{t}\int^1_0\mathbf udz+\text{div}_{\bf x}\int^1_0(\mathbf{u}\otimes\mathbf{u})dz+\int^1_0\partial_z(\mathbf u w)dz+\nabla_{\bf x}\int^1_0pdz=0. \\
\end{align*}
Utilizing the boundary condition \eqref{1.3} and then taking the divergence, we get
\begin{align*}
-\Delta_{\bf x}p=\text{div}_{\bf x}\text{div}_{\bf x}\int^1_0(\mathbf{u}\otimes\mathbf{u})dz.
\end{align*}
This is an elliptic problem similar as in \cite{b2,bo}. According to the methods therein, we can obtain our desired result.

\end{proof}

Now we turn to prove Theorem \ref{t2.1}.  Similar to \cite{b2}, we consider the test function $\Psi=\chi(t)\phi(\mathbf x,z)$ with compact support in $(t_a,t_b)\times S_\phi
\subset\subset(t_1,t_2)\times\widetilde{\Omega}$, and introduce the open sets $Q_i$ satisfying \eqref{2.4} and the corresponding mollifiers \eqref{2.10}. Clearly $\chi(t)\phi(\mathbf x,z)\mathbf u(\mathbf x,z,t)\in L^3(Q_1)$ has its support in $Q_3$. We extend it to the whole space:
\begin{align*}
\overline{\chi(t)\phi(\mathbf x,z)\mathbf u(\mathbf x,z,t)}=\chi(t)\phi(\mathbf x,z)\overline{\mathbf u(\mathbf x,z,t)}\in L^3(\mathbb{R}_t\times\mathbb{R}^2_{\mathbf x}\times\mathbb{R}_z)
\end{align*}
Then the extension is regularized according to
\begin{align*}
\Psi_{\epsilon,\kappa}=\rho_{\epsilon,\kappa}\star\left(\chi(t)\phi(\mathbf x,z)(\rho_{\epsilon,\kappa}\star\overline{\mathbf u})(t,\mathbf x,z)\right)
=:\left(\chi(t)\psi(\mathbf x,z)(\overline{\mathbf u(t,\mathbf x,z)})^{\epsilon,\kappa}\right)^{\epsilon,\kappa}\in\mathcal{D}(\mathbb{R}_t\times\mathbb{R}^2_{\mathbf x}\times\mathbb{R}_z),
\end{align*}
where $\epsilon\in(0,\frac{\eta}{2})$, $\kappa\in(0,\frac{\tau}{2})$. Here we choose $\Psi_{\epsilon,\kappa}$ as the test function in the definition of PE in \eqref{2.1}, and have the formula
\begin{align*}
\langle\langle\partial_{t}\mathbf u+\text{div}_{\bf x}(\mathbf{u}\otimes\mathbf{u})+\partial_z(\mathbf u w)+\nabla_{\bf x}p, \Psi_{\epsilon,\kappa}\rangle\rangle=0.
\end{align*}
Then we denote $I^{\epsilon,\kappa}_1=\langle\langle\partial_t\mathbf u, \Psi_{\epsilon,\kappa}\rangle\rangle$,
$I^{\epsilon,\kappa}_2=\langle\langle\text{div}_{\bf x}(\mathbf{u}\otimes\mathbf{u}), \Psi_{\epsilon,\kappa}\rangle\rangle$,
$I^{\epsilon,\kappa}_3=\langle\langle\partial_z(\mathbf u w), \Psi_{\epsilon,\kappa}\rangle\rangle$ and
$I^{\epsilon,\kappa}_4=\langle\langle\nabla_{\bf x}p, \Psi_{\epsilon,\kappa}\rangle\rangle$. The terms $I^{\epsilon,\kappa}_1$ and $I^{\epsilon,\kappa}_2$ can be calculated the same way as in \cite{b2}, we just state the result:
\begin{align}\label{3.1}
I^{\epsilon,\kappa}_1=\int_{\mathbb{R}_t\times\mathbb{R}^2_{\mathbf x}\times\mathbb{R}_z}(\partial_t\overline{\mathbf u})^{\epsilon,\kappa}
\left(\chi(t)\psi(\mathbf x,z)\overline{\mathbf u(t,\mathbf x,z)}\right)^{\epsilon,\kappa}d\mathbf xdzdt,
\end{align}
and
\begin{align}\label{3.2}
I^{\epsilon,\kappa}_2&=-\int_{\mathbb{R}_t}\chi(t)
\int_{\mathbb{R}^2_{\mathbf x}\times\mathbb{R}_z}[(\overline{\mathbf u}\otimes\overline{\mathbf u})^{\epsilon,\kappa}-
\left((\overline{\mathbf u})^{\epsilon,\kappa}\otimes(\overline{\mathbf u})^{\epsilon,\kappa}\right)]:\nabla_{\mathbf x}(\psi(\mathbf x,z)(\overline{\mathbf u})^{\epsilon,\kappa})d\mathbf xdzdt\nonumber\\
&-\int_{\mathbb{R}_t}\chi(t)
\int_{\mathbb{R}^2_{\mathbf x}\times\mathbb{R}_z}\frac{|(\overline{\mathbf u})^{\epsilon,\kappa}|^2}{2}(\overline{\mathbf u})^{\epsilon,\kappa}\cdot\nabla_{\mathbf x}\psi(\mathbf x,z)d\mathbf xdzdt\nonumber\\
&-\int_{\mathbb{R}_t}\chi(t)
\int_{\mathbb{R}^2_{\mathbf x}\times\mathbb{R}_z}\text{div}_{\mathbf x}(\overline{\mathbf u})^{\epsilon,\kappa}\frac{|(\overline{\mathbf u})^{\epsilon,\kappa}|^2}{2}\psi(\mathbf x,z)d\mathbf xdzdt.
\end{align}
For $I^{\epsilon,\kappa}_{3}$, we have
\begin{align}\label{3.3}
I^{\epsilon,\kappa}_3&=\langle\langle\partial_z(\mathbf u w), \Psi_{\epsilon,\kappa}\rangle\rangle
=-\int_{Q_2}\mathbf uw\partial_z\Psi_{\epsilon,\kappa}d\mathbf Xdt
=-\int_{\mathbb{R}_t\times\mathbb{R}^3_{\mathbf X}}\overline{\mathbf u}\overline{w}\partial_z\Psi_{\epsilon,\kappa}d\mathbf Xdt\nonumber\\
&=-\int_{\mathbb{R}_t\times\mathbb{R}^2_{\mathbf x}\times\mathbb{R}_z}\overline{\mathbf u}\overline{w}\partial_z\left(\chi(t)\psi(\mathbf x,z)(\overline{\mathbf u})^{\epsilon,\kappa}\right)^{\epsilon,\kappa}d\mathbf xdzdt\nonumber\\
&=-\int_{\mathbb{R}_t\times\mathbb{R}^2_{\mathbf x}\times\mathbb{R}_z}
[(\overline{\mathbf u}\overline{w})^{\epsilon,\kappa}-(\overline{\mathbf u})^{\epsilon,\kappa}(\overline{w})^{\epsilon,\kappa})]\cdot
\partial_{z}\left(\chi(t)\psi(\mathbf x,z)(\overline{\mathbf u})^{\epsilon,\kappa}\right)d\mathbf xdzdt\nonumber\\
&\hspace{6pt}-\int_{\mathbb{R}_t\times\mathbb{R}^2_{\mathbf x}\times\mathbb{R}_z}
(\overline{\mathbf u})^{\epsilon,\kappa}(\overline{w})^{\epsilon,\kappa}\cdot
\partial_{z}\left(\chi(t)\psi(\mathbf x,z)(\overline{\mathbf u})^{\epsilon,\kappa}\right)d\mathbf xdzdt.
\end{align}
Then, we integrate by parts in the second term on the right hand side of \eqref{3.3} and get
\begin{align}\label{3.4}
-\int_{\mathbb{R}_t\times\mathbb{R}^2_{\mathbf x}\times\mathbb{R}_z}
&(\overline{\mathbf u})^{\epsilon,\kappa}(\overline{w})^{\epsilon,\kappa}\cdot
\partial_{z}\left(\chi(t)\psi(\mathbf x,z)(\overline{\mathbf u})^{\epsilon,\kappa}\right)d\mathbf xdzdt\nonumber\\
&=-\int_{\mathbb{R}_t\times\mathbb{R}^2_{\mathbf x}\times\mathbb{R}_z}
\frac{|(\overline{\mathbf u})^{\epsilon,\kappa}|^2}{2}(\overline{w}^{\epsilon,\kappa})\partial_{z}\left(\chi(t)\psi(\mathbf x,z)\right)d\mathbf xdzdt\nonumber\\
&\hspace{6pt}+\int_{\mathbb{R}_t\times\mathbb{R}^2_{\mathbf x}\times\mathbb{R}_z}\frac{|(\overline{\mathbf u})^{\epsilon,\kappa}|^2}{2}\chi(t)\psi(\mathbf x,z)\partial_z(\overline{w})^{\epsilon,\kappa}d\mathbf xdzdt.
\end{align}
We combine the third term of \eqref{3.2} and the second term of \eqref{3.4} to obtain
\begin{align}\label{3.5}
\int_{\mathbb{R}_t}\chi(t)
\int_{\mathbb{R}^2_{\mathbf x}\times\mathbb{R}_z}
\left(\text{div}_{\mathbf x}(\overline{\mathbf u})^{\epsilon,\kappa}+\partial_z(\overline{w})^{\epsilon,\kappa}\right)
\frac{|(\overline{\mathbf u})^{\epsilon,\kappa}|^2}{2}\psi(\mathbf x,z)d\mathbf xdzdt\nonumber\\
=\int_{\mathbb{R}_t}\chi(t)
\int_{\mathbb{R}^2_{\mathbf x}\times\mathbb{R}_z}
\left(\overline{\text{div}_{\mathbf x}\mathbf u+\partial_zw}\right)^{\epsilon,\kappa}
\frac{|(\overline{\mathbf u})^{\epsilon,\kappa}|^2}{2}\psi(\mathbf x,z)d\mathbf xdzdt=0,
\end{align}
where we used Proposition \ref{pro2.1} and the incompressibility condition.

By the same token, for $I^{\epsilon,\kappa}_{4}$, we have
\begin{align}\label{3.6}
I^{\epsilon,\kappa}_4&=\langle\langle\nabla_{\mathbf x}p, \Psi_{\epsilon,\kappa}\rangle\rangle
=-\int_{\mathbb{R}_t\times\mathbb{R}^2_{\mathbf x}\times\mathbb{R}_z}\overline{p}\text{div}_{\mathbf x}\left(\chi(t)\psi(\mathbf x,z)(\overline{\mathbf u})^{\epsilon,\kappa}\right)^{\epsilon,\kappa}d\mathbf Xdt\nonumber\\
&=-\int_{\mathbb{R}_t}\chi(t)\int_{\mathbb{R}^2_{\mathbf x}\times\mathbb{R}_z}
(\overline{p})^{\epsilon,\kappa}\left((\overline{\mathbf u})^{\epsilon,\kappa}\cdot\nabla_{\mathbf x}\psi
+\text{div}_{\mathbf x}(\overline{\mathbf u})^{\epsilon,\kappa}\psi\right)d{\bf x} dzdt\nonumber\\
&=I^{\epsilon,\kappa}_{41}+I^{\epsilon,\kappa}_{42}.
\end{align}
Using Proposition \ref{pro2.1} again, it is evident to obtain
\begin{align}\label{3.7}
I^{\epsilon,\kappa}_{42}=
&=-\int_{\mathbb{R}_t}\chi(t)\int_{\mathbb{R}^2_{\mathbf x}\times\mathbb{R}_z}
(\overline{p})^{\epsilon,\kappa}\text{div}_{\mathbf x}(\overline{\mathbf u})^{\epsilon,\kappa}\psi d\mathbf xdzdt
=-\int_{\mathbb{R}_t}\chi(t)\int_{\mathbb{R}^2_{\mathbf x}\times\mathbb{R}_z}
(\overline{p})^{\epsilon,\kappa}(\overline{\text{div}_{\mathbf x}\mathbf u})^{\epsilon,\kappa}\psi d\mathbf xdzdt\nonumber\\
&=\int_{\mathbb{R}_t}\chi(t)\int_{\mathbb{R}^2_{\mathbf x}\times\mathbb{R}_z}
(\overline{p})^{\epsilon,\kappa}(\overline{\partial_z w})^{\epsilon,\kappa}\psi d\mathbf xdzdt
=\int_{\mathbb{R}_t}\chi(t)\int_{\mathbb{R}^2_{\mathbf x}\times\mathbb{R}_z}
(\overline{p})^{\epsilon,\kappa}\partial_z (\overline{w})^{\epsilon,\kappa}\psi d\mathbf xdzdt\nonumber\\
&=-\int_{\mathbb{R}_t}\chi(t)\int_{\mathbb{R}^2_{\mathbf x}\times\mathbb{R}_z}
(\overline{p})^{\epsilon,\kappa}(\overline{w})^{\epsilon,\kappa}\partial_z \psi d\mathbf xdzdt.
\end{align}
Summing up all the above equations together, we get
\begin{align}\label{3.8}
\int_{\mathbb{R}_t\times\mathbb{R}^2_{\mathbf x}\times\mathbb{R}_z}
&\left[\frac{|(\overline{\mathbf u})^{\epsilon,\kappa}|^2}{2}\partial_t(\psi(\mathbf x,z)\chi(t))
+\left(\frac{|(\overline{\mathbf u})^{\epsilon,\kappa}|^2}{2}(\overline{\mathbf u})^{\epsilon,\kappa}+(\overline{p})^{\epsilon,\kappa}(\overline{\mathbf u})^{\epsilon,\kappa}\right)\nabla_{\mathbf x}(\psi(\mathbf x,z)\chi(t))\right.\nonumber\\
&+\left.\left(\frac{|(\overline{\mathbf u})^{\epsilon,\kappa}|^2}{2}+(\overline{p})^{\epsilon,\kappa}\right)(\overline{w})^{\epsilon,\kappa}\partial_z(\chi(t)\psi(\mathbf x,z))\right]d{\bf x}dzdt\nonumber\\
&=-\int_{\mathbb{R}_t\times\mathbb{R}^2_{\mathbf x}\times\mathbb{R}_z}[(\overline{\mathbf u}\otimes\overline{\mathbf u})^{\epsilon,\kappa}-
\left((\overline{\mathbf u})^{\epsilon,\kappa}\otimes(\overline{\mathbf u})^{\epsilon,\kappa}\right)]:\nabla_{\mathbf x}(\chi(t)\psi(\mathbf x,z)(\overline{\mathbf u})^{\epsilon,\kappa})d\mathbf xdzdt\nonumber\\
&\hspace{6pt}-\int_{\mathbb{R}_t\times\mathbb{R}^2_{\mathbf x}\times\mathbb{R}_z}
[(\overline{\mathbf u}\overline{w})^{\epsilon,\kappa}-(\overline{\mathbf u})^{\epsilon,\kappa}(\overline{w})^{\epsilon,\kappa})]\cdot
\partial_{z}\left(\chi(t)\psi(\mathbf x,z)(\overline{\mathbf u})^{\epsilon,\kappa}\right)d\mathbf xdzdt.
\end{align}
Letting $\kappa\rightarrow0$ in \eqref{3.8} and recalling Proposition~\ref{pressureest}, one obtains that
\begin{align}\label{3.9}
\int_{\mathbb{R}_t\times\mathbb{R}^2_{\mathbf x}\times\mathbb{R}_z}
&\big[\frac{|(\overline{\mathbf u})^{\epsilon}|^2}{2}\partial_t(\psi(\mathbf x,z)\chi(t))
+\left(\frac{|(\overline{\mathbf u})^{\epsilon}|^2}{2}(\overline{\mathbf u})^{\epsilon}+(\overline{p})^{\epsilon}(\overline{\mathbf u})^{\epsilon}\right)\nabla_{\mathbf x}(\psi(\mathbf x,z)\chi(t))\nonumber\\
&+\left(\frac{|(\overline{\mathbf u})^{\epsilon}|^2}{2}+(\overline{p})^{\epsilon}\right)(\overline{w})^{\epsilon}\partial_z(\chi(t)\psi(\mathbf x,z))\big]\nonumber\\
&=-\int_{\mathbb{R}_t}\chi(t)\int_{\mathbb{R}^2_{\mathbf x}\times\mathbb{R}_z}[(\overline{\mathbf u}\otimes\overline{\mathbf u})^{\epsilon}-
\left((\overline{\mathbf u})^{\epsilon}\otimes(\overline{\mathbf u})^{\epsilon}\right)]:\nabla_{\mathbf x}(\psi(\mathbf x,z)(\overline{\mathbf u})^{\epsilon})d\mathbf xdzdt\nonumber\\
&-\int_{\mathbb{R}_t}\chi(t)\int_{\mathbb{R}^2_{\mathbf x}\times\mathbb{R}_z}
[(\overline{\mathbf u}\overline{w})^{\epsilon}-(\overline{\mathbf u})^{\epsilon}(\overline{w})^{\epsilon})]\cdot
\partial_{z}\left(\psi(\mathbf x,z)(\overline{\mathbf u})^{\epsilon}\right)d\mathbf xdzdt.
\end{align}

Let us first demonstrate the convergence to zero, as $\epsilon\to0$, of the error terms on the right hand side. Owing to the well-known commutator estimate from \cite{co} in the anisotropic H\"older space,
\begin{align}\label{3.10}
\int^{t_2}_{t_1}\chi(t)&|\langle(\overline{\mathbf u})^\epsilon\otimes(\overline{\mathbf u})^\epsilon-(\overline{\mathbf u\otimes\mathbf u})^\epsilon,\nabla_{\mathbf x}(\psi(\overline{\mathbf u})^\epsilon)\rangle|d\mathbf xdzdt\nonumber\\
&\leq C(\chi,\psi)(\epsilon^{3\alpha-1}+\epsilon^{\alpha+2\beta-1})
\int^{t_2}_{t_1}\|\mathbf u\|^3_{C^{0,\alpha}_{\mathbf x}C^{0,\beta}_{z}(\overline{\widetilde{\Omega}})}dt.
\end{align}
This converges to zero as $\epsilon\to0$ under the stated assumptions.

For the second commutator in~\eqref{3.9}, first observe in analogy with \cite{b2} that
\begin{align}\label{3.11}
\partial_z\big(\phi(\mathbf x,z)(\overline{\mathbf u})^\epsilon\big)
&=\partial_z\psi(\overline{\mathbf u})^\epsilon+\psi\partial_z(\overline{\mathbf u})^\epsilon\nonumber\\
&=\partial_z\psi(\overline{\mathbf u})^\epsilon
+\psi\int_{\mathbb{R}^2_{\mathbf x}\times\mathbb{R}_z}\partial_z\rho_\epsilon(\mathbf x,z-y)I_{2}(y)\overline{\mathbf u}(y)dyd{\bf x}\nonumber\\
&=\partial_z\psi(\overline{\mathbf u})^\epsilon
+\psi\int_{\mathbb{R}^2_{\mathbf x}\times\mathbb{R}_z}\partial_z\rho_\epsilon(\mathbf x,z-y)\big(I_2(z)\overline{\mathbf u}(z)-I_{2}(y)\overline{\mathbf u}(y)\big)dyd{\bf x},
\end{align}
whence we deduce
\begin{align}\label{3.12}
|\partial_z\big(\phi(\mathbf x,z)(\overline{\mathbf u})^\epsilon\big)|\leq C\epsilon^{\beta-1}\|\mathbf u\|_{C^{0,\alpha}_{\bf x}C^{0,\beta}_z}.
\end{align}
As in \cite{co}, we have
\begin{align}
(\overline{\mathbf u})^\epsilon(\overline{w})^\epsilon-(\overline{\mathbf u}\overline{w})^\epsilon
=\left(\overline{\mathbf u}-(\overline{\mathbf u})^\epsilon\right)
\left(\overline{w}-(\overline{w})^\epsilon\right)
-\int\delta_y\overline{\mathbf u}\delta_y\overline{w}\rho_\epsilon(y)dy,
\end{align}
where $\delta_y\overline{\mathbf u}=\overline{\mathbf u}(\mathbf x,z-y)-\overline{\mathbf u}(\mathbf x,z)$. We only treat the first term as the second is analogous.

By direct calculation, we have $|\overline{\mathbf u}-(\overline{\mathbf u}^\epsilon)|\leq C(\epsilon^\alpha+\epsilon^\beta)\|\mathbf u\|_{C^{0,\alpha}_xC^{0,\beta}_z}$.


Recalling the incompressibility condition, we have $w=\int^z_0\text{div}_{\mathbf x}\mathbf u(\mathbf x,s)ds$ and
\begin{align*}
\big{(}\overline{w}-(\overline{w}^\epsilon)\big{)}(\mathbf X)
&=\int_{\mathbb{R}^2_{\mathbf x}\times\mathbb{R}_z}\big{(}1-\rho_\epsilon\big{)}(\mathbf Y)\overline{w}(\mathbf X-\mathbf Y)d\mathbf Y\\
&=\int_{\mathbb{R}^2_{\mathbf x}\times\mathbb{R}_z}\big{(}1-\rho_\epsilon\big{)}(\mathbf Y)
\left(\int^{z-y_3}_0\text{div}_{\mathbf x}\overline{\mathbf u}(\mathbf x-\mathbf y,s)ds\right)
d\mathbf Y\\
&=\int_{\mathbb{R}_z}dy_3\int^{z-y_3}_0ds
\int_{\mathbb{R}^2_{\mathbf x}}\big{(}1-\rho_\epsilon\big{)}(\mathbf Y)
\text{div}_{\mathbf x}\overline{\mathbf u}(\mathbf x-\mathbf y,s)d\mathbf y
\end{align*}
where $\mathbf X=(\mathbf x,z)$ and $\mathbf Y=(\mathbf y,y_3)$.

Therefore, it is obvious to have the following through integration by parts
\begin{align*}
|\overline{w}-(\overline{w}^\epsilon)|
&\leq|\int_{\mathbb{R}_z}dy_3\int^{z-y_3}_0ds
\int_{\mathbb{R}^2_{\mathbf x}}\nabla_{\mathbf y}\rho_\epsilon(\mathbf Y)\cdot\overline{\mathbf u}(\mathbf x-\mathbf y,s)d\mathbf y|\\
&=|\int_{\mathbb{R}_z}dy_3\int^{z-y_3}_0ds
\int_{\mathbb{R}^2_{\mathbf x}}\nabla_{\mathbf y}\rho_\epsilon(\mathbf Y)\cdot
\left(\overline{\mathbf u}(\mathbf x,s)-\overline{\mathbf u}(\mathbf x-\mathbf y,s)\right)d\mathbf y|\\
&\leq C\epsilon^{\alpha-1}\|\mathbf u\|_{C^{0,\alpha}_{\bf x}C^{0,\beta}_z},
\end{align*}
where we have used that $\mathbf u\in L^3((0,T);C^{0,\alpha}_{\bf x}C^{0,\beta}_z)$ and
$\int_{\mathbb{R}^2_{\mathbf x}\times\mathbb{R}_z}|\nabla_{\mathbf y}\rho_{\epsilon}(\mathbf Y)|d\mathbf Y\leq C\epsilon^{-1}$.

Putting these estimates together, we get
\begin{align}\label{3.14}
\int^{t_2}_{t_1}\chi(t)&|\langle(\overline{\mathbf u}\overline{w})^{\epsilon}-(\overline{\mathbf u})^{\epsilon}(\overline{w})^{\epsilon},\partial_{z}(\psi(\mathbf x,z)(\overline{\mathbf u})^{\epsilon})\rangle|dt\nonumber\\
&\leq C(\chi,\psi)\epsilon^{\beta-1}(\epsilon^\alpha+\epsilon^\beta)\epsilon^{\alpha-1}\int^{t_2}_{t_1}\|\mathbf u\|^3_{C^{0,\alpha}_{\mathbf x}C^{0,\beta}_{z}(\overline{\widetilde{\Omega}})}dt.
\end{align}
Depending on whether $\alpha<\beta$ or not, we obtain the requirements $2\alpha+\beta>2$ or $\alpha+2\beta>2$, respectively.
Finally, we consider the left hand side of~\eqref{3.9}. As $\epsilon\to0$, it converges to the expression in the local energy equality, tested against $\chi\psi$.

\qed

\begin{remark}
Let us briefly discuss the cases $\alpha>1$ or $\beta>1$. In \cite{bo}, it is required that both the horizontal and the vertical velocities belong to $L^4((0,T);B^{\alpha,\infty}_{4}(\mathbb{T}^3))$ ($\frac{1}{2}<\alpha<1$) in Proposition 3.1, and to
$L^3((0,T);B^{\beta,\infty}_{3}(\mathbb{T};B^{\alpha,\infty}_{3}(\mathbb{T}^2))$ ($\alpha>1,\beta>\frac{1}{3}$, $\alpha+2\beta>2$) in Proposition 3.8.

%
%

{ 
In comparison, we too can decrease the vertical regularity if we increase the horizontal one. More precisely, {considering $\mathbf u\in C^{1,\alpha-1}_{\bf x}C^{0,\beta}_z (1<\alpha<2)$ and} taking into account the incompressibility condition $w=\int^z_0\operatorname{div}_{\mathbf x}\mathbf u(\mathbf x,s)ds$, we deduce directly $w \in C^{0,\alpha-1}_{\bf x}C^{0,\beta+1}_z$, which implies that
\begin{align*}
|\overline{w}-(\overline{w}^\epsilon)|
\leq C\epsilon^{\alpha-1}\|w\|_{C^{0,\alpha-1}_{\bf x}C^{0,\beta+1}_z}.
\end{align*}
We thus obtain
\begin{align*}
\int^{t_2}_{t_1}\chi(t)&|\langle(\overline{\mathbf u}\overline{w})^{\epsilon}-(\overline{\mathbf u})^{\epsilon}(\overline{w})^{\epsilon},\partial_{z}(\psi(\mathbf x,z)(\overline{\mathbf u})^{\epsilon})\rangle|dt\nonumber\\
&\leq C(\chi,\psi)\epsilon^{\alpha+2\beta-2}
\int^{t_2}_{t_1}\|\mathbf u\|^2_{C^{0,\alpha}_{\mathbf x}C^{0,\beta}_{z}(\overline{\widetilde{\Omega}})}\|w\|_{C^{0,\alpha-1}_{\bf x}C^{0,\beta+1}_z}dt.
\end{align*}
We can get the same conclusion as Theorem \ref{t2.1}, only requiring $\beta\in (0, 1/2)$, $\alpha+2\beta>2$.

{In contrast}, it is interesting to find that, due to~\eqref{3.10}, the horizontal regularity always needs to stay greater than $1/3$, even if we increase the vertical regularity.

This shows that the anisotropic relation between the vertical and horizontal directions plays an important role in the problem of energy conservation, as already pointed out in~\cite{bo}.

}

\end{remark}

\section{Proof of Theorem \ref{t2.2}}\label{S4}
In order to obtain the global conservation of energy, we must control the boundary condition. Under the assumption that the boundary is a $C^2$ manifold, Bardos, Titi, and Wiedemann~\cite{b2} consider the distance function
\begin{align*}
d(\mathbf X)=d(\mathbf X,\partial\Omega)=\inf_{\mathbf Y\in\partial\Omega}|\mathbf X-\mathbf Y|\geq0,
\hspace{6pt}\text{and}\hspace{5pt} V_{\eta_0}=\{\mathbf X\in\Omega, d(\mathbf X)<\eta_0\}.
\end{align*}
It is known~\cite{foote} that $d|_{V_{\eta_0}}\in C^2(\overline{V_{\eta_0}})$ and for any $\mathbf X\in V_{\eta_0}$, there exists a unique $\sigma(\mathbf X)\in \partial\Omega$ such that
\begin{align*}
d(\mathbf X)=|\mathbf X-\sigma(\mathbf X)|, \quad\quad \nabla d(\mathbf X)=-\mathbf n(\sigma(\mathbf X)).
\end{align*}

However, we should emphasize that our domain is a cylinder and its boundary is not $C^2$. Therefore, the method used in \cite{b2} can not be straightforwardly applied to our proof. Here we regularize the boundary, smoothing the corner point, so that the boundary is $C^2$ and one may safely employ the distance function.

We consider a subset $\widetilde{\Omega}\subset\subset\Omega$ such that $\Omega\setminus\widetilde{\Omega}\subset\subset V_{\eta_0}$. And we assume $\Omega'$ is a domain $\widetilde{\Omega}\subset\subset\Omega'\subset\subset\Omega$. From the hypothesis \eqref{2.15} and Proposition 3.1, we deduce
\begin{align*}
\mathbf u\otimes\mathbf u\in L^{\frac{3}{2}}((t_1,t_2); C^{0,\alpha}_{\mathbf x}C^{0,\beta}_{z}(\overline{\Omega'})),\hspace{5pt}
p\in L^{\frac{3}{2}}((t_1,t_2); C^{0,\alpha}_{\mathbf x}C^{0,\beta}_{z}(\overline{\Omega'})).
\end{align*}
Thanks to Theorem \ref{t2.1}, one obtains
\begin{align*}
\frac{d}{dt}\langle\frac{|\mathbf u|^2}{2},\psi\rangle-\langle\left(\frac{|\mathbf u|^2}{2}+p\right)\mathbf u,\nabla_{\mathbf x}\psi\rangle-\langle\left(\frac{|\mathbf u|^2}{2}+p\right)w,\partial_z\psi\rangle=0
\end{align*}
for any $\psi\in C^1_c(\widetilde{\Omega})$.

For $\eta>0$ sufficiently small, we denote
$S_{\eta}=\{\bm{x}\in S:\ {\rm dist}_{{\mathbb R}^2}
(\bm{x},\partial S)
<\eta\}$. As $\partial S\in C^2$, then so is $\partial S_\eta$. Thus for $\bm{x}\in S_{\eta}\setminus S_{2\eta}$,
there exist $\bm{x}^{(0)}\in \partial S$,
$\bm{x}^{(1)}\in \partial S_\eta$ and
$\bm{x}^{(2)}\in \partial S_{2\eta}$ such that $\bm{x}$
and $\bm{x}^{(j)}$, $j=0,1,2$ are all on the line with direction
$\bm{n}(\sigma(\bm{x}))$, which is the normal direction of $\partial S$
at $\bm{x}^{(0)}$. It is also the normal direction of
$\partial S_{j\eta}$ at $\bm{x}^{(j)}$, $j=1,2$.

Choose a strictly increasing function $\tau:[0,\eta]\to[\eta,2\eta]$ that is $C^2$ on $(0,\eta)$ and such that $\tau(0)=\eta$, $\tau'(0)=\tau''(0)=0$ and $\tau(\eta)=2\eta$, $(\tau^{-1})'(2\eta)=(\tau^{-1})''(2\eta)=0$, and
set
\begin{align}
\omega_1(\bm{x})&=\tau\left(|\bm{x}-\bm{x}^{(2)}|\right),\\
\omega_2(\bm{x})&={1-\omega_1(\bm{x})}
\end{align}
for $\bm{x}\in S_{\eta}\setminus S_{2\eta}$.
Then we define the domain $\Omega_\eta\subset
S_{\eta}\times(\eta,1-\eta)$
as follows:
\begin{equation}\label{4.3}
z\in
\left\{\begin{array}{ll}
(\eta,1-\eta),\ & {\rm if }\
\bm{x}\in S_{2\eta},\\
(\omega_1(\bm{x}),\omega_2(\bm{x})),\ & {\rm if }\
\bm{x}\in S_{\eta}\setminus S_{2\eta}.
\end{array}
\right.
\end{equation}
It can be checked that $\partial\Omega_{\eta}\in C^2$  (see Figure~2). Then we introduce a non-negative $C^\infty$ function
\begin{equation}
\phi(s)=
\left\{
\begin{array}{ll}
1, & s>\frac{1}{2},\\
0, & s<\frac{1}{4},
\end{array}\right.
\end{equation}
satisfying $|\phi'|\le 10$ and consider the non-negative function
$d_{\eta}(\bm{X})={\rm dist}(\bm{X},\partial\Omega_\eta)$, which is $C^2$ near $\partial\Omega_\eta$.
We define $\psi_{\eta}(\bm{X})
=\phi(\frac{d_\eta(\bm{X}))}{\eta})\in C_c^2(\Omega)$, so that
\begin{equation}
\nabla\psi_{\eta}(\bm{X})=
\left\{
\begin{array}{ll}
-\frac{1}{\eta}\phi'(\frac{d_\eta(\bm{X})}{\eta})
\bm{N}(\bm{\sigma}(\bm{X})),
& \frac{\eta}{4}<d(\bm{X})<\frac{\eta}{2},\\
0, & {\rm otherwise }.
\end{array}
\right.
\end{equation}

Furthermore, we have
\begin{equation}
\nabla\psi_{\eta}(\bm{X})=
\left\{
\begin{array}{ll}
\frac{1}{\eta}\phi'(\frac{d_\eta(\bm{X})}{\eta})
\bm{e}_3, & \bm{X}\in S_{2\eta}\times
(\frac{5}{4}\eta,\frac{3}{2}\eta),\\
-\frac{1}{\eta}\phi'(\frac{d_\eta(\bm{X})}{\eta})
\bm{e}_3, & \bm{X}\in S_{2\eta}\times
(1-\frac{3}{2}\eta,1-\frac{5}{4}\eta),\\
-\frac{1}{\eta}\phi'(\frac{d_\eta(\bm{X})}{\eta})
\left(\bm{n}(\bm{\sigma}(\bm{X})),0\right), & \bm{X}\in
(S_{\frac{5}{4}\eta}
\setminus S_{\frac{3}{2}\eta})\times
(2\eta,1-2\eta).
\end{array}
\right.
\end{equation}
Then we have
\begin{align}
&\int_\Omega\frac{|\bm{u}(t_2,\mathbf x,z)|^2}{2}\psi_\eta d\bm{x}dz
-\int_\Omega\frac{|\bm{u}(t_1,\bm{x},z)|^2}{2}\psi_\eta d\bm{x}dz\nonumber\\
&=-\int^{t_2}_{t_1}\int_{\{\frac{\eta}{4}<d_\eta(\bm{X})<\frac{\eta}{2}\}}(\frac{|\bm{u}|^2}{2}+p)
\bm{U}(t,\bm{X})\cdot
\bm{N}(\sigma(\bm{X}))\frac{1}{\eta}\phi'd\bm{X} dt.
\end{align}
From the statement of the theorem, recall
\begin{equation*}
O=\left\{\frac{\eta}{4}<d_\eta(\bm{X})<\frac{\eta}{2}\right\}
\setminus\left\{\left[S_{2\eta}\times
\left[\left(\frac{5}{4}\eta,\frac{3}{2}\eta\right)\cup
\left(1-\frac{3}{2}\eta,1-\frac{5}{4}\eta\right)\right]
\right]
\cup
\left[(S_{\frac{5}{4}\eta}
\setminus S_{\frac{3}{2}\eta})\times
(2\eta,1-2\eta)\right]
\right\}.
\end{equation*}
Obviously, $|O|=O(\eta^2)$ and $\|\nabla\psi_{\eta}(\bm{X})\|
\le 10\eta^{-1}$ for $\bm{X}\in O$.
Taking the limit, we obtain
\begin{align*}
&\lim_{\eta\to 0+}\left|
\int^{t_2}_{t_1}\int_{\{\frac{\eta}{4}<d_\eta(\bm{X})<\frac{\eta}{2}\}}(\frac{|\bm{u}|^2}{2}+p)
\bm{U}(t,\bm{X})\cdot
\bm{N}(\sigma(\bm{X}))\frac{1}{\eta}\phi'd\bm{X} dt
\right|\nonumber\\
\le&\lim_{\eta\to 0+}10\eta^{-1}\left\{
\int^{t_2}_{t_1}\int_{S_{\frac{5}{4}\eta}\setminus S_{\frac{3}{2}\eta}}
\int_{2\eta}^{1-2\eta}\left|(\frac{|\bm{u}|^2}{2}+p)\bm{u}\cdot
\bm{n}(\bm{\sigma}(\bm{X}))\right|dz
d\bm{x}dt+\int^{t_2}_{t_1}\int_{S_{2\eta}}
\int_{I_\eta}\left|(\frac{|\bm{u}|^2}{2}+p)w\right|
dz d\bm{x}dt\right.\nonumber\\
&\left.
+\int^{t_2}_{t_1}\int_{O}\left|(\frac{|\bm{u}|^2}{2}+p)|\bm{U}|\right|
d\bm{X}dt
\right\}\nonumber\\
\le&10\lim_{\eta\to 0+}\eta^{-1}\left\{
\int^{t_2}_{t_1}\int_{[S_{\frac{5}{4}\eta}\setminus S_{\frac{3}{2}\eta}]
\times[0,1]}\left|(\frac{|\bm{u}|^2}{2}+p)\bm{u}\cdot
\bm{n}(\bm{\sigma}(\bm{X}))\right|
d\bm{X}dt+\int^{t_2}_{t_1}\int_{S\times I_\eta}
\left|(\frac{|\bm{u}|^2}{2}+p)w\right|
d\bm{X}dt\right\}\nonumber\\
&+O(\eta^{3\mu_2-1})+O(\eta^{\mu_1+\mu_2-1})=0.
\end{align*}
\qed


\section{Vanishing Viscosity Limit}\label{S5}
The inviscid limit problem in a domain with physical boundary is notoriously difficult due to the possible formation of a turbulent boundary layer. In general, it is unknown whether a sequence of Leray-Hopf solutions will converge to a solution of the inviscid problem in any sense.
One approach is to prove conditional results under a uniform assumption on the viscosity sequence,
as in Bardos et al.~\cite{b2}, who establish the viscosity limit and exclude anomalous dissipation of the limit solution under the assumption that the sufficient conditions for energy conservation hold uniformly in viscosity.

Although there has been much progress in the past years concerning the vanishing viscosity of limit of Navier-Stokes equations, few result address the PE system. To the authors' knowledge, only Kukavica et al.~\cite{k} consider the zero viscosity limit for analytic solutions in a bounded domain. By virtue of the aforementioned discussion, we can give a sufficient condition of vanishing viscosity for the PE system.

\begin{theorem}\label{t5.1}
Let $(\mathbf u_\nu, w_\nu,p_\nu)$ be a family of Leray-Hopf weak solutions of the viscous PE system
\begin{equation}\label{5.1}
\left\{
\begin{array}{llll}
\partial_{t}\mathbf u_\nu+\operatorname{div}_{\bf x}(\mathbf{u}_\nu\otimes\mathbf{u}_\nu)+\partial_z(\mathbf u_\nu w_\nu)+\nabla_xp_\nu
=\nu\Delta_{\mathbf x}\mathbf u+\nu\partial_{zz}\mathbf u, \\
\partial_zp_\nu=0,\\
\operatorname{div}_{\bf x}\mathbf u_\nu+\partial_zw_\nu=0,
\end{array}\right.
\end{equation}
with initial data $\mathbf u_0$ (independent of $\nu$) and boundary conditions
\begin{align}
\mathbf u_\nu|_{\Gamma_s\cup\Gamma_b}=0, \hspace{5pt}w|_{\Gamma_t}=0,
\hspace{5pt}\nu\partial_z\mathbf u_\nu|_{\Gamma_t}=0.
\end{align}
Moreover, we assume $\mathbf u_\nu$ satisfies the hypotheses of Theorem~\ref{t2.2} uniformly in $\nu$, specifically:

1. There exists some $\eta_0>0$, an open subset $V_{\eta_0}=\{\mathbf X\in\Omega: d(\mathbf X,\partial\Omega)<\eta_0\}$ and $\theta<\infty$ such that
\begin{align}\label{5.3}
\sup_{\nu>0}\|p_\nu\|_{L^{\frac{3}{2}}\big{(}(0,T);H^{-\theta(V_{\eta_0})}(V_{\eta_0})\big{)}}<\infty;
\end{align}

2. For any $(t_1,t_2)\times\widetilde{\Omega}\subset(0,T)\times\Omega$, $\alpha\in (1/2,1)$ and {\color{blue} $\beta\in (1/2,1)$ } such that $2\operatorname{min}\{\alpha,\beta\}+\operatorname{max}\{\alpha,\beta\}>2$ with the property
\begin{align}\label{5.4}
\sup_{\nu>0}\int^{t_2}_{t_1}\|\mathbf{u}_\nu\|^3_{C^{0,\alpha}_{\mathbf x}C^{0,\beta}_{z}(\overline{\widetilde{\Omega}})}dt\leq M(\widetilde{\Omega})<\infty;
\end{align}

3.
\begin{align}\label{5.5}
\lim_{\eta\rightarrow0}\sup_{\nu>0}\int^T_0\frac{1}{\eta}\left\{
\int_{[S_{\frac{5}{4}\eta}\setminus S_{\frac{3}{2}\eta}]
\times[0,1]}\left|(\frac{|\bm{u}_\nu|^2}{2}+p_\nu)\bm{u}_\nu\cdot
\bm{n}(\bm{\sigma}(\bm{X}))\right|
d\bm{X}+\int_{S\times I_\eta}
\left|(\frac{|\bm{u}_\nu|^2}{2}+p_\nu)w_\nu\right|
d\bm{X}
\right\}dt=0;
\end{align}

4. There exist some $\mu_1,\mu_2\in\mathbb{R}$ such that
\begin{align}\label{auxest}
&\sup_{\nu>0}\|p_\nu\|_{L^\frac{3}{2}((0,T)\times\Gamma_\eta)}\leq C \eta^{\mu_1},\hspace{8pt}
\sup_{\nu>0}\|\bm{U}_\nu\|_{L^3((0,T)\times \Gamma_\eta)}\leq C\eta^{\mu_2},\nonumber\\
&\mu_1+\mu_2>1,\ \mu_2>\frac{1}{3}.
\end{align}

Then (extracting a subsequence $\nu$ if necessary) $\mathbf u_\nu$ converges weakly* in $L^\infty((0,T);L^2(\Omega))$ to a function $\overline{\mathbf u}_\nu\in C_{weak}([0,T);L^2(\Omega))$ which is a weak solution of PE system \eqref{5.1} with the same initial data and which also satisfies the hypotheses of Theorem 2.2.  Moreover, there is no anomalous energy dissipation in the vanishing viscosity limit, i.e., for every $T^\star\in(0,T)$ one has
\begin{align}\label{vandiss}
\lim_{\nu\rightarrow0}\nu\int^{T^\star}_0\int_{\Omega}|\nabla_{\mathbf x}\mathbf u_\nu|^2+|\partial_{zz}\mathbf u_\nu|^2d\mathbf xdzdt=0.
\end{align}

\end{theorem}

\begin{proof}
Since $\mathbf u_\nu$ is a Leray-Hopf weak solution of the viscous PE system \eqref{5.1}, it satisfies the energy inequality:
\begin{align}\label{5.6}
\frac{1}{2}\int_\Omega\|\mathbf u_\nu(t)\|^2_{L^2}d\mathbf xdz+\nu\int^{t}_0\int_{\Omega}\big{(}|\nabla_{\mathbf x}\mathbf u_\nu|^2+|\partial_{zz}\mathbf u_\nu|^2\big{)}d\mathbf xdzds
\leq\frac{1}{2}\int_\Omega\|\mathbf u_0\|^2_{L^2}d\mathbf xdz.
\end{align}
From the Banach-Alaoglu Theorem, there exists a subsequence $\mathbf u_{\nu_j}$ in $L^\infty((0,T);L^2(\Omega))$ converging to $\overline{\mathbf u_\nu}\in C_{weak}([0,T);L^2(\Omega))$, and $w_{\nu_j}$ converges weakly to $\overline{w_{\nu}}$ in $L^2((0,T);L^2(\Omega))$. Therefore, it is easy to deduce
\begin{align}\label{5.7}
&\text{div}_{\mathbf x}\overline{\mathbf u_\nu}+\partial_z\overline{w_\nu}=0\hspace{5pt} \text{in}\hspace{3pt} (0,T)\times{\Omega},\nonumber\\
&\overline{\mathbf u_\nu}\cdot\mathbf n|_{\Gamma_s}=0,\hspace{3pt}\overline{w_\nu}|_{\Gamma_b\cup\Gamma_t}=0,\nonumber\\
&\lim_{t\rightarrow0}\int_\Omega\overline{\mathbf u_\nu(t)}\phi(\mathbf X)d\mathbf X=\int_\Omega\mathbf u_0(\mathbf X)\phi(\mathbf X)d\mathbf X\quad\quad \text{for all $\phi\in L^2(\Omega)$.}
\end{align}

From the energy inequality \eqref{5.6}, we get $\|\overline{\mathbf u_\nu}(t)\|_{L^2}\leq\|\mathbf u_0\|_{L^2}$ for $t\geq0$. By virtue of weak lower semicontinuity of the norm and \eqref{5.7}, we get
\begin{align}\label{5.8}
\lim_{t\rightarrow0}\|\overline{\mathbf u_\nu}(t)-\mathbf u_0\|_{L^2(\Omega)}=0.
\end{align}

For any $\widetilde{\Omega}\subset\subset{\Omega}$, we have $\mathbf u_\nu\in L^3((0,T);C^{0,\alpha}_{\mathbf x}C^{0,\beta}_z)$ due to \eqref{5.4}. Thanks to the momentum equation $\eqref{5.1}$
and \eqref{5.6}, one deduces $\partial_t\mathbf u_\nu\in L^{\frac{3}{2}}((0,T);H^{-1}(\widetilde{\Omega}))$. Therefore, we can prove up to a subsequence (which we denote as $\mathbf u_{\nu_j}$) $\mathbf u_{\nu_j}\rightarrow\overline{\mathbf u_\nu}$ in $L^2((0,T)\times\widetilde{\Omega})$ by virtue of the Aubin-Lions Lemma. Then we choose open subsets $\widetilde{\Omega_k}\subset\widetilde{\Omega_{k+1}}$ such that $\cup\widetilde{\Omega_k}=\Omega$ and extract a diagonal subsequence, which converges in $L^2((0,T)\times\Omega)$ to the limiting function $\overline{\mathbf u_\nu}$. Moreover, we know $w_{\nu_j}\rightharpoonup\overline{w_\nu}$ in $L^2((0,T)\times\Omega)$. Therefore, $(\overline{\mathbf u_\nu},\overline{w_\nu})$ is a solution of the PE system.

In the following, we need to prove the solution $\overline{\mathbf u_\nu},\overline{w_\nu}$ satisfies:
\begin{align}\label{5.9}
\lim_{\eta\rightarrow0}\int^T_0\frac{1}{\eta}\left\{
\int_{[S_{\frac{5}{4}\eta}\setminus S_{\frac{3}{2}\eta}]
\times[0,1]}\left|(\frac{|\overline{\bm{u}_\nu}|^2}{2}+\overline{p_\nu})\overline{\bm{u}_\nu}\cdot
\bm{n}(\bm{\sigma}(\bm{X}))\right|
d\bm{X}+\int_{S\times I_\eta}
\left|(\frac{|\overline{\bm{u}_\nu}|^2}{2}+\overline{p_\nu})\overline{w_\nu}\right|
d\bm{X}
\right\}dt=0.
\end{align}
Utilizing \eqref{5.4} and Proposition \ref{pressureest}, we get
\begin{align*}
\int^T_0\left|(\frac{|\mathbf u_\nu|^2}{2}+p_\nu)\mathbf u_\nu(t,\mathbf X)\mathbf n(\sigma(\mathbf X))
\right|+\left|(\frac{|\mathbf u_\nu|^2}{2}+p_\nu)w_\nu\right| dt
\end{align*}
is equicontinuous with respect to $\mathbf X$ in every compact subset of $\Omega$. Then for every $\eta>0$ and for every $\mathbf X$ satisfying $\frac{\eta}{4}<d(\mathbf X)<\frac{\eta}{2}$, we have

\begin{align}\label{5.10}
\lim_{\nu_j\rightarrow0}&\int^T_0\left|(\frac{|\mathbf u_{\nu_j}|^2}{2}+p_{\nu_j})\mathbf u_{\nu_j}(t,\mathbf X)\mathbf n(\sigma(\mathbf X))
\right|+\left|(\frac{|\mathbf u_{\nu_j}|^2}{2}+p_{\nu_j})w_{\nu_j}\right|dt\nonumber\\
&=\int^T_0\left|(\frac{|\overline{\mathbf u_{\nu}}|^2}{2}+\overline{p_{\nu}})\overline{\mathbf u_{\nu}}(t,\mathbf x)\mathbf n(\sigma(\mathbf x))
\right|+\left|(\frac{|\overline{\mathbf u_{\nu}}|^2}{2}+\overline{p_{\nu}})\overline{w_{\nu}}\right|dt,
\end{align}
where we used the Arzel\`a-Ascoli Theorem after passing to a diagonal subsequence if necessary.

It is obvious to obtain from \eqref{5.10}
\begin{align*}
\lim_{\nu_j\rightarrow0}&\int^T_0\frac{1}{\eta}\left\{\int_{[S_{\frac{5}{4}\eta}\setminus S_{\frac{3}{2}\eta}]
\times[0,1]}
|(\frac{|\mathbf u_{\nu_j}|^2}{2}+p_{\nu_j})\mathbf u_{\nu_j}(t,\mathbf x)\mathbf n(\sigma(\mathbf x))|d\mathbf X
+\int_{S\times I_\eta}|(\frac{|\mathbf u_{\nu_j}|^2}{2}+p_{\nu_j})w_{\nu_j}| d\mathbf X\right\} dt\\
&=\int^T_0\frac{1}{\eta}\left\{\int_{[S_{\frac{5}{4}\eta}\setminus S_{\frac{3}{2}\eta}]
\times[0,1]}
|(\frac{|\overline{\mathbf u_{\nu}}|^2}{2}+\overline{p_{\nu}})\overline{\mathbf u_{\nu}}(t,\mathbf x)\mathbf n(\sigma(\mathbf x))|d\mathbf X
+\int_{S\times I_\eta}|(\frac{|\overline{\mathbf u_{\nu}}|^2}{2}+\overline{p_{\nu}})\overline{w_{\nu}}| d\mathbf X\right\} dt.
\end{align*}
Therefore, we can use \eqref{5.5} to imply \eqref{5.9}. The properties~\eqref{auxest} for the limit solution can be established similarly, so that the latter satisfies the hypotheses of Theorem~\ref{t2.2} and thus conserves the energy. The remaining conclusion \eqref{vandiss} then follows analogously to \cite{b2}.

\end{proof}



\vskip 0.5cm
\noindent {\bf Acknowledgements}

\vskip 0.1cm
The research of  \v{S}. Ne\v{c}asov\'{a} has been supported by the Czech Science Foundation (GA\v CR) project 22-01591S and Praemium Academiae of \v S. Ne\v casov\' a. T. Tang is partially supported by NSFC No. 12371246, NSF of Jiangsu Province Grant No. BK20221369 and Qing Lan Project of Jiangsu Province. E. Wiedemann is supported by the DFG Priority Programme SPP 2410 (project number 525716336).
L. Zhu is partially supported by NSFC No. 12001162 and the Fundamental Research Funds for the Central University No. 2013/B220202082. T.~Tang wishes to thank the University of Erlangen-N\"{u}rnberg for their hospitality during his stay in Germany. In turn, E.~Wiedemann expresses his gratitude to Yangzhou University for hosting him.


\end{document}